\documentclass[letterpaper,11pt]{amsart}

\usepackage[all]{xy}                      %

\CompileMatrices                            

\UseTips                                    

\input xypic
\usepackage[bookmarks=true]{hyperref}       

\usepackage{amssymb,latexsym,amsmath,amscd}
\usepackage{xspace}
\usepackage{color}
\usepackage{kpfonts}
\usepackage{graphicx}
\usepackage{dsfont}

\reversemarginpar

\theoremstyle{plain}
\newtheorem{theorem}{Theorem}[section]
\newtheorem*{theorem*}{Theorem}
\newtheorem{proposition}[theorem]{Proposition}
\newtheorem{corollary}[theorem]{Corollary}
\newtheorem{lemma}[theorem]{Lemma}

\theoremstyle{definition}
\newtheorem{definition}[theorem]{Definition}

\newtheorem{remark}[theorem]{Remark}
\newtheorem{example}[theorem]{Example}

\newcommand{\enm}[1]{\ensuremath{#1}}          %
\newcommand{\op}[1]{\operatorname{#1}}
\newcommand{\cal}[1]{\mathcal{#1}}

\newcommand{\ZZ}{\enm{\mathbb{Z}}}

\newcommand{\PP}{\enm{\mathbb{P}}}

\newcommand{\Aa}{\enm{\cal{A}}}

\newcommand{\Ee}{\enm{\cal{E}}}
\newcommand{\Ff}{\enm{\cal{F}}}
\newcommand{\Gg}{\enm{\cal{G}}}

\newcommand{\Ii}{\enm{\cal{I}}}

\newcommand{\Kk}{\enm{\cal{K}}}
\newcommand{\Ll}{\enm{\cal{L}}}

\newcommand{\Oo}{\enm{\cal{O}}}

\renewcommand{\phi}{\varphi}
\renewcommand{\theta}{\vartheta}
\renewcommand{\epsilon}{\varepsilon}

\newcommand{\Pic}{\op{Pic}}

\newcommand{\Ext}{\op{Ext}}

\renewcommand{\to}[1][]{\xrightarrow{\ #1\ }}

\newcommand{\old}[1]{}

\begin{document}

\title[aCM sheaves of pure rank two on reducible hyperquadrics]{aCM sheaves of pure rank two \\on reducible hyperquadrics}

\author{Edoardo Ballico, Sukmoon Huh and Joan Pons-Llopis}

\address{Universit\`a di Trento, 38123 Povo (TN), Italy}
\email{edoardo.ballico@unitn.it}

\address{Sungkyunkwan University, Suwon 440-746, Korea}
\email{sukmoonh@skku.edu}

\address{Department of Mathematics, Kyoto University, Kyoto, Japan}
\email{ponsllopis@gmail.com}

\thanks{The first author is partially supported by MIUR and GNSAGA of INDAM (Italy). The second author is supported by Basic Science Research Program 2015-037157 through NRF funded by MEST and the National Research Foundation of Korea(KRF) 2016R1A5A1008055 grant funded by the Korea government(MSIP). The third author is supported by a FY2015 JSPS Postdoctoral Fellowship}


\begin{abstract}
We classify a special type of arithmetically Cohen-Macaulay sheaves of rank two on reducible and reduced quadric hypersurfaces. As a consequence we show that a reducible and reduced quadric surface is of wild type.
\end{abstract}

\maketitle


\section{Introduction}
The Horrocks theorem states that a vector bundle $\Ee$ on the projective space $\PP^n$ splits as the sum of line bundles if and only if $\Ee$ has {\it no intermediate cohomology}, that is, $H^i(\PP^n, \Ee(t))=0$ for any $t\in \ZZ$ and $i=1,\ldots, n-1$. It has stimulated the study of {arithmetically Cohen-Macaulay} (for short, aCM) vector bundles that are bundles with no intermediate cohomology supported on a given projective variety $X$, because in the algebraic context these bundles correspond to maximal Cohen-Macaulay modules over the associated graded ring under the assumption that $X$ is aCM, i.e. its associated graded ring is Cohen-Macaulay.

The classification of aCM vector bundles has been done, partially in some cases, for several projective varieties such as smooth quadric hypersurfaces \cite{Kapranov, Knorrer}, cubic surfaces \cite{CH}, Fano threefolds \cite{Madonna} and others. In \cite{DG} a classification of aCM varieties was proposed as {\it finite, tame or wild} representation type according to the complexity of their associated category of aCM vector bundles and there are several contributions to this trichotomy such as \cite{EH,BGS,CMP,FM}. Recently this problem is solved in \cite{FP}, except when the varieties are cones.

In this article we investigate aCM sheaves of rank up to two on a reducible and reduced quadric hypersurfaces $X_n$ in $\PP^{n+1}$ with $n\ge 2$, which are a particular case of cones. Notice that the one-dimensional case, namely the union of two lines $C:=L_1\cup L_2\subset\PP^2$ is already understood: indeed, in \cite{ESW}, it is explained how to construct Ulrich sheaves on $C$; on the other hand, in \cite{DG} it was provided a complete classification of aCM sheaves on $C$. Our first main result is on $X=X_2$ that is a union of two distinct planes $H_i\subset \PP^3$, $i=1,2$ whose intersection is a line $L$:

\begin{theorem}
Any aCM sheaf of pure rank one on $X$ is an extension of a line bundle on one plane by another line bundle on the other plane.
\end{theorem}

Then we consider the case of aCM sheaves of rank two. Indeed, in Proposition \ref{s+1} aCM bundles of rank two turn out to be given as kernel sheaves of simple type, and so it is natural to classify aCM kernel sheaves of simple type with rank two.

\begin{theorem}\label{main}
Let $\Ee$ be an aCM kernel sheaf of simple type with pure rank two on $X$. Up to twist it is one of the following:
\begin{itemize}
\item a direct sum of two line bundles;
\item an extension of a twisted ideal sheaf $\Ii_p(1)$ of a point $p\not\in L$ by $\Oo_X$, which is locally free;
\item a sheaf whose restriction to each component of $X$ satisfies
\begin{itemize}
\item[(i)] $\Ee_{|H_i} \cong \Oo_{H_i}(c)\oplus \Oo_{H_i}$ and
\item[(ii)] $0\to \Oo_{H_j}(k) \to \Ee_{|H_j} \to \Ii_{Z, H_j}(c-k) \to 0$
\end{itemize}
for an integer $c$ at least $2$ and $\{i,j\}=\{1,2\}$, where $Z\subset L$ is a zero-dimensional subscheme with $k=|Z|$ such that $0\le k<c\le 2k+2$.
\end{itemize}
\end{theorem}
Then concerning the trichotomy classification of aCM varieties, we use the classification of aCM sheaves of rank one and two on $X$ to obtain:

\begin{corollary}
$X$ is of wild type in a very strong sense, that is, there are arbitrarily large dimensional families of pairwise non-isomorphic aCM sheaves of rank one and two on $X$.
\end{corollary}

Note that a smooth quadric surface is of finite type, while a quadric surface with a unique singular point is CM countable, i.e. there are only countably many isomorphism classes of indecomposable aCM sheaves; see \cite{BGS}.

Then we turn our interest on aCM sheaves of rank two on a higher dimensional quadric hypersurface, and we get the following conclusion using the results on $X$.

\begin{theorem}
If $\Ee$ is an aCM kernel sheaf of simple type with pure rank two on $X_n$ with $n\ge 3$, then it is decomposable.
\end{theorem}

Let us summarize here the structure of this paper. In section $2$ we introduce the definitions and main properties that will be used throughout the paper, mainly those related to aCM sheaves and $m$-regularity of coherent sheaves. In section $3$ we give a complete description of aCM sheaves of pure rank one and collect several technical lemmas to investigate aCM bundles. We introduce the notion of a kernel sheaf of simple type as a generalization of aCM bundle. Then we show that every aCM bundle of rank two on $X_n$ with $n\ge 3$ splits. In section $4$ we suggest essentially two types of aCM sheaves of rank two on $X$ and show that $X$ is of wild type. Finally in section $5$ we classify all the aCM kernel sheaves of simple type of rank two on $X$ and show that every aCM sheaves of simple type of rank two on $X_n$ with $n \ge 3$ splits.

\section{Preliminaries}
Throughout the article, our base field $\mathbf{k}$ is algebraically closed of characteristic $0$.

In this section we are going to introduce definitions and properties of $m$-regularity and arithmetically Cohen-Macaulay sheaves that are going to be used throughout the rest of the paper.

Let $X$ be a projective variety with an ample line bundle $\Oo_X(1)$. For a coherent sheaf $\Ee$ on $X$ and $t\in \ZZ$, let us denote $\Ee\otimes \Oo_X(t)$ by $\Ee(t)$. As usual, $H^i(X,\Ee)$ stands for the cohomology group, $h^i(X,\Ee)$ for its dimension. If there is no confusion, we will skip $X$, i.e. we will write $H^i(\Ee)$ and $h^i(\Ee)$. We also denote the dual of $\Ee$ by $\Ee^\vee$. The ideal sheaf of a subscheme $Z\subset X$ is denoted by $\Ii_{Z,X}$, or simply $\Ii_Z$ if $X$ is a reducible and reduced quadric surface or there is no confusion.

\begin{definition}
A coherent sheaf $\Ee$ on a projective variety $X$ with an ample line bundle $\Oo_X(1)$ is called {\it arithmetically Cohen-Macaulay} (for short, aCM) if it is locally Cohen-Macaulay, i.e. $\mathrm{depth }~\Ee_x=\dim \Oo_{X,x}$ for all $x\in X$ and $H^i(\Ee(t))=0$ for all $t\in \ZZ$ and $i=1, \ldots, \dim (X)-1$.
\end{definition}
Recall that $X$ is said to be aCM if its homogeneous coordinate ring $\mathbf{k}[X]$ is a Cohen-Macaulay ring, that is, $\mathrm{depth}\mathbf{k}[X]=\dim \mathbf{k}[X]$. ACM sheaves on an aCM variety $X$ are important, because they are in one-to-one correspondence with the maximal Cohen-Macaulay modules over $\mathbf{k}[X]$. Inspired by a classification for quivers and for $\mathbf{k}$-algebras of finite type, the following trichotomy classification of aCM varieties was proposed in \cite{DG}:
\begin{itemize}
\item $X$ is of {\it finite type} if there are only finitely many indecomposable aCM sheaves up to twist;
\item $X$ is of {\it tame type} if for each rank $r$, the indecomposable aCM sheaves of rank $r$ form a finite number of families of dimension at most $1$;
\item $X$ is of {\it wild type} if there exist families of arbitrarily large dimension of indecomposable pairwise non-isomorphic aCM sheaves.
\end{itemize}
In \cite{FP} the complete trichotomy classification was accomplished for aCM varieties that are not cones.

Another important object in the study of aCM sheaves are Ulrich sheaves.
\begin{definition}\label{ulr}
An aCM sheaf $\Ee$ on $X\subset \PP^n$ with $h^0(\Ee(-1))=0$ is said to be an {\it Ulrich sheaf} if we have $h^0(\Ee)=\deg (X)\mathrm{rank}(\Ee)$.
\end{definition}
In other words, an Ulrich sheaf is the one achieving the maximal possibility for the minimal number of generators of $\mathbf{k}[X]$-module $\oplus_{t\in \ZZ} H^0(\Ee(t))$; see \cite{CMP}.

Now we introduce an important notion in projective algebraic geometry that will be used throughout the paper.
\begin{definition}
A coherent sheaf $\Ee$ on a projective variety $X$ with an ample line bundle $\Oo_X(1)$ is called {\it $m$-regular} if $H^i(\Ee (m-i))=0$ for all $i>0$.
\end{definition}
Due to the following result, the regularity measures the point at which cohomological complexities vanish.
\begin{theorem}\cite[4D at page 67]{Eisenbud}\cite[Theorem 1.8.5]{Lazarsfeld}
If $\Ee$ is an $m$-regular sheaf on $X$ with respect to $\Oo_X(1)$, then for any $k\ge0$,
\begin{itemize}
\item[(i)] $\Ee(m+k)$ is globally generated,
\item[(ii)] the natural maps
$$H^0(\Ee(m))\otimes H^0(\Oo_X(k)) \to H^0(\Ee(m+k))$$
are surjective, and
\item[(iii)] $\Ee$ is $(m+k)$-regular.
\end{itemize}
\end{theorem}
In particular, if $\Ee$ is $0$-regular, then it is globally generated. Since each bundle $\Ee$ is $m$-regular for some $m\gg0$, it follows that there is a twist of $\Ee$ which is $0$-regular. Note that a sheaf being an aCM is invariant under twist and so we may always assume that our aCM sheaf is $0$-regular and so it is globally generated.


\section{Bundles on a reducible quadric surface}

Let $X\subset \PP^3$ be a reducible and reduced quadric surface, i.e. $X = H_1\cup H_2$ with $H_i$ distinct planes in $\PP^3$. Let $S=\mathbf{k}[x,y,z,w]$ be the polynomial ring of $\PP^3$ and $x,y$ the defining equations of $H_1, H_2$, respectively. If we define $L:=H_1 \cap H_2$, then we have $\Pic (X)\cong \ZZ$ by \cite[Example 5.2]{HartK}, generated by a hyperplane section, which in general consists of lines $L_i \subset H_i$ for $i=1,2$, meeting $L$ at the same point. Let us denote the ample generator $\Oo_{\PP^3}(1) \otimes \Oo_X$ of $\Pic (X)$ by $\Oo_X(1)$. Then the dualizing sheaf of $X$ is $\omega_X \cong \Oo_X(-2)$.

\begin{definition}
A coherent sheaf $\Ee$ on $X$ is said to have pure rank $r\in \ZZ$ if it has rank $r$ at a general point on each component $H_i$.
\end{definition}

\begin{remark}
There are two other reasonable notions of rank:
\begin{enumerate}
\item the pair $(r_1, r_2)$ of two ranks at general points of each components; they may be different for sheaves not locally free at any point on $\Ll$,
\item the Hilbert polynomial gives a unique number in $(1/2)\ZZ$, which is equal to $(r_1+r_2)/2$.
\end{enumerate}
\end{remark}

\begin{remark}\label{prem}
Since every line bundle on a projective plane is aCM, so is every line bundle on $X$ due to the following exact sequence
\begin{equation}\label{yyy}
0\to \Oo_X \to \Oo_{H_1}\oplus \Oo_{H_2} \to \Oo_L \to 0,
\end{equation}
because the map $H^0(\Oo_{H_1}(t)\oplus \Oo_{H_2}(t)) \rightarrow H^0(\Oo_L(t))$ is surjective for each $t\in \ZZ$. As a consequence, every direct sum of line bundles on $X$ is aCM. Note also that any extension of an aCM bundle by a line bundle splits.
\end{remark}


\begin{lemma}\label{ee1}
Let $\Ee$ and $\Ff$ be two vector bundles on $X$ such that
\begin{itemize}
\item[(i)] $\Ee _{|H_1}$ is a direct sum of line bundles, and
\item[(ii)] $\Ee _{|H_i} \cong \Ff _{|H_i}$ for each $i=1,2$.
\end{itemize}
Then we have $\Ee \cong \Ff$.
\end{lemma}

\begin{proof}
Consider the exact sequence
\begin{equation}\label{eqee1}
0 \to \Ee \otimes \Ff ^\vee \to  \Ee \otimes \Ff ^\vee _{|H_1}\oplus  \Ee \otimes \Ff ^\vee _{|H_2}\to  \Ee \otimes \Ff ^\vee _{|L}\to 0.
\end{equation}
Fix an isomorphism $g: \Ee _{|H_2} \rightarrow \Ff _{|H_2}$ and let $g': \Ee _{|L}\rightarrow \Ff _{|L}$ denote its restriction to $L$. By (\ref{eqee1}) it is sufficient to prove the existence of an isomorphism $f: \Ee _{|H_1}\rightarrow \Ff _{|H_1}$ such that $f_{|L} = g'$, because it would imply the existence of a morphism $\Ee\rightarrow \Ff$ whose restriction to $H_i$ is an isomorphism for each $i=1,2$.

Fix homogeneous coordinates $[x_0:x_1:x_2]$ on $H_1$ such that $L :\{x_0=0\}$. We identify $\Ee _{|H_1}$ and $\Ff _{|H_1}$ with the split bundle $\Aa := \oplus  _{i=1}^{r}\Oo _{H_1}(a_i)$ and their restriction to $L$ with $\Aa _{|L}$. With these identification $g'$ is given by an $(r\times r)$-matrix $M$ with homogeneous linear entries in $x_1,x_2$. The same matrix $M$ induces a morphism $f: \Ee_{|H_1} \cong \Aa \rightarrow  \Aa\cong \Ff_{|H_1}$ such that $f_{|L} =g'$ by definition. Now it is enough to prove that $f$ is invertible, i.e. $\det (M)$ vanishes at no point of $H_1$. Note that $\det (M)$ is a homogeneous form in the variables $x_1,x_2$. Since $M$ has no zeros on $L$ and our base field is algebraically closed, then $\det (M)$ is a nonzero constant and so $f$ is invertible at each point of $H_1$.
\end{proof}


Tensoring (\ref{yyy}) with a vector bundle $\Ee$ on $X$, we also have an exact sequence
\begin{equation}\label{eqa2}
0 \to \Ee \to \Ee_{|H_1}\oplus \Ee_{|H_2} \to \Ee_{|L} \to 0.
\end{equation}

Fix a positive integer $r$ and let $\Ff_i$ be a vector bundle of rank $r$ on $H_i$ with the restriction map $u_i: \Ff_i\rightarrow \Ff_{i|L}$ for each $i=1,2$. Assume $\Ff_{1|L}\cong \Ff_{2|L}$ and fix an isomorphism $e: \Ff_{2|L}\rightarrow \Ff_{1|L}$ as $\Oo _L$-modules. Two maps $u_1$ and $e \circ u_2$ may be seen as maps of $\Oo_X$-sheaves, considering $\Oo_{H_i}$ as a quotient of $\Oo_X$. Then we get a surjection $u:=(e\circ u_1,-u_2): \Ff_1\oplus \Ff_2 \rightarrow \Ff_{1|L}$ of $\Oo _X$-sheaves with $\Kk := \mathrm{ker}(u)$ a coherent $\Oo _X$-sheaf. By definition we have an exact sequence
\begin{equation}\label{eqx1}
0 \to \Kk \to \Ff_1\oplus \Ff_2 \to \Ff_{1|L}\to 0.
\end{equation}

\begin{definition}\label{yyt}
We call such a sheaf $\Kk$ above {\it a kernel sheaf}. If one of the bundles $\Ff_i$ splits, then $\Kk$ is said to be {\it of simple type}.
\end{definition}

The restriction $\Kk_{|H_i}$ may have torsion on some points on $L$ and its quotient by torsion part $\Kk_{|H_i}^{\circ}$ is isomorphic to $\Ff_i$ for each $i$.

\begin{lemma}\label{rrr0}
Let $\Ee$ be an aCM sheaf of rank two fitting into an exact sequence
\begin{equation}\label{eqrrr11}
0 \to \Ee \stackrel{j}{\to} \Ee _1\oplus \Ee _2 \to \Ee _{1|L}\to 0
\end{equation}
with each $\Ee _i$ a vector bundle of rank two on $H_i$ and $\Ee _1\cong \Oo _{H_1}(a)\oplus \Oo _{H_1}(b)$ for some integer $a\ge b$. Then $H^0(\Ee (-b))$ spans $\Ee (-b)$ at all points of $X\setminus L$ and $\Ee_2(-b)$ is spanned at all points of $H_2\setminus L$.
\end{lemma}

\begin{proof}
It is enough to deal with the case $b=0$. Let $V$ be the kernel of the restriction map $ H^0(H_1,\Ee _1)\rightarrow H^0(L,\Ee _{1|L})$. Then by (\ref{eqrrr11}) $V$ sits inside $H^0(\Ee )$ and it spans
$\Ee$ at all point of $H_1\setminus L$. Since $h^1(L,\Ee _{1|L}(-1)) =0$ and $\Ee$ is aCM, we have $h^1(H_2,\Ee _2(-1)) =0$. Hence the restriction map $H^0(H_2,\Ee _2)\rightarrow H^0(L,\Ee_{2|L})$ is surjective. Thus (\ref{eqrrr11}) gives the existence of a linear subspace $W\subset H^0(\Ee)$ mapping isomorphically onto $H^0(H_1,\Ee _1)$. Since the restriction map $H^0(H_1,\Ee _1) \rightarrow H^0(L,\Ee _{1|L})$ is surjective, there is a linear subspace $W\subseteq H^0(\Ee)$ mapped isomorphically onto $H^0(H_2,\Ee _2)$. Since the map $j$ induces an isomorphism on $X\setminus L$, $H^0(\Ee)$ spans $\Ee$ at all points of $H_2\setminus L$ and $\Ee_2$ is spanned at all points of $H_2\setminus L$.
\end{proof}

\begin{remark}\label{rrr1}
Note that for aCM kernel sheaf $\Ee$ of simple type with rank two, the integer $t$ for which $\Ee_{|H_i}(t)$ is spanned with a trivial factor, is uniquely determined by Lemma \ref{rrr0}: it is the maximal integer $t$ such that $\Ee (-t)$ is spanned at a general point of $H_1$ and at a general point of $H_2$.
\end{remark}

\begin{remark}
We do not claim that the kernel sheaf $\Kk$ is independent on the choice of the isomorphisms $e: \Ff_{2|L}\rightarrow \Ff_{1|L}$; see \cite{Daoudi}.
\end{remark}

\begin{lemma}\label{x2}
With notations as above, if $\Ff_1\cong \oplus _{i=1}^{r} \Oo _{H_1}(a_i)$ for $a_i \in \ZZ$, then we have
$$h^1(\Kk(t))=\dim \ker \left\{ H^1(\Ff_2(t)) \to H^1({\Ff_2}_{|L}(t)) \right\}$$
for each $t\in \ZZ$.
\end{lemma}

\begin{proof}
By assumption we have $h^1(\Ff_1(t)) =0$ for all $t\in \ZZ$ and the restriction map $H^0(\Ff_1(t)) \rightarrow H^0(\Ff_{1|L}(t))$ is surjective. Now we may use (\ref{eqa2}) for $\Ee=\Kk(t)$.
\end{proof}

\begin{proposition}\label{propp}
If $\Ee$ is an aCM sheaf of pure rank one on $X$, then it fits into an exact sequence
\begin{equation}\label{ssee}
0\to \Oo_{H_i}(a)\to \Ee \to \Oo_{H_{3-i}}(b)\to 0
\end{equation}
for some $i\in \{1,2\}$ and $a,b\in \ZZ$.
\end{proposition}
\begin{proof}
We may assume that $\Ee$ is $0$-regular, but not $(-1)$-regular. Thus we have $h^2(\Ee(-3))>0$, which gives a nonzero map $u: \Ee \rightarrow \Oo_X(1)$. Let $\mathrm{Im}(u)=\Ii_A(1)$ for some closed subscheme $A\subsetneq X$. Since $\Ee$ is globally generated, $\Ii_A(1)$ is also globally generated. Moreover, if $A\ne \emptyset= A \cap L$, then $A$ is a zero-dimensional subscheme and $u$ is an isomorphism. But since $h^1(\Ii_A(-1))>0$, we must have $A\cap L\ne \emptyset$. Thus $A$ is one of the following:
\begin{itemize}
\item [(i)] $A=\emptyset$;
\item [(ii)] $A = L_1\cup L_2$ with $L_i$ a line of $H_i$ such that $L_1\cap L_2\ne \emptyset$,
\item [(iii)] $A=D \ne L$ for a line $D\subset H_i$ for some $i\in \{1,2\}$,
\item [(iv)] $A=L$,
\item [(v)] $A =\{p\}$ with $p\in L$,
\item [(vi)] $A$ is a connected $0$-dimensional subscheme of degree $2$ such that $Z_{\mathrm{red}}\in L$,
\item [(vii)] $A\supseteq H_i$ for some $i\in \{1,2\}$.
\end{itemize}
In each case but (vii), the surjective map $u:\Ee \rightarrow \Ii_A(1)$ is an isomorphism. In case (i) $\Ee$ is isomorphic to $\Oo_X(1)$ and so it fits into the sequence
\begin{equation}\label{seee}
0\to \Oo_{H_2}(-1) \to \Oo_X \to \Oo_{H_1} \to 0
\end{equation}
twisted by $\Oo_X(1)$. In case (ii) we get $\Ee \cong \Oo_X$. In case (iii) without loss of generality we may assume that $D\subset H_1$. Then we see from (\ref{seee}) that the ideal sheaf $\Ii_D$ fits into the sequence
$$0\to \Oo_{H_2}(-1) \to \Ii_D \to \Ii_{D, H_1} \cong \Oo_{H_1}(-1) \to 0.$$
In case (iv) we get the same extension as in (iii). In cases (v) and (vi), we have $h^1(\Ee(t))>0$ for $t<-1$ and so it is not aCM. Now in case (vii), without loss of generality we assume that $A\supset H_2$. Since $\Ii_A(1)$ is globally generated, we get $A=H_2$ scheme-theoretically and so $\Ii_A(1) \cong \Oo_{H_1}$. So we have $h^1(\mathrm{ker}(u)(t))=0$ for $t<0$. On the other hand, from $h^2(\Ee(-2))=0$ we get $h^2(\mathrm{ker}(u)(-2))=0$ and so $\mathrm{ker}(u)$ is $0$-regular, giving the other vanishing. Thus $\mathrm{ker}(u)$ is an aCM sheaf of rank one supporting on $H_2$ and so we get $\mathrm{ker}(u) \cong \Oo_{H_2}(a)$ for some $a\in \ZZ$.
\end{proof}

\begin{remark}
Without loss of generality, consider the extension (\ref{ssee}) with $i=2$ and $a=-1$. If $b>0$, then (\ref{ssee}) splits. If $b=0$, we get $\Ee \cong \Oo_X$. If $b<0$, then we get $\Ee \cong \Ii_C$ for a plane curve $C\subset H_1$ of degree $-b$. Indeed, if we apply the functor $\mathrm{Hom}_{\PP^3}(\Oo_{H_1}(b), -)$ to the standard exact sequence for $H_2\subset \PP^3$ twisted by $\Oo_{\PP^3}(-1)$, we get
\begin{align*}
0&\to \mathrm{Hom}_{\PP^3}(\Oo_{H_1}(b), \Oo_{H_2}(-1)) \to \Ext_{\PP^3}^1(\Oo_{H_1}(b), \Oo_{\PP^3}(-2)) \\
&\to \Ext_{\PP^3}^1(\Oo_{H_1}(b), \Oo_{\PP^3}(-1))\to \Ext_{\PP^3}^1(\Oo_{H_1}(b), \Oo_{H_2}(-1)) \to 0
\end{align*}
since $\Ext_{\PP^3}^2(\Oo_{H_1}(b), \Oo_{\PP^3}(-2)) \cong H^1(\Oo_{H_1}(b-2))^\vee=0$. Similarly we get
$$\Ext_{\PP^3}^1(\Oo_{H_1}(b), \Oo_{\PP^3}(-2))\cong H^2(\Oo_{H_1}(b-2))^\vee \cong H^0(\Oo_{H_1}(-b-1)),$$
which is the dimension of $\mathrm{Hom}_{\PP^3}(\Oo_{H_1}(b), \Oo_{H_2}(-1))$. Thus we get the isomorphism
$$\Ext_{\PP^3}^1(\Oo_{H_1}(b), \Oo_{H_2}(-1))\cong \Ext_{\PP^3}^1(\Oo_{H_1}(b), \Oo_{\PP^3}(-1)) \cong H^0(\Oo_{H_1}(-b)).$$
Conversely, for any plane curve $C\subset H_1$ of degree $-b$, its ideal sheaf $\Ii_C$ is an extension of $\Oo_{H_1}(-b)$ by $\Oo_{H_2}(-1)$.
\end{remark}

\begin{proposition}\label{s+1}
Let $\Ee$ be be an aCM bundle of rank two. Then $\Ee$ is a kernel sheaf of simple type.
\end{proposition}

\begin{proof}
Since $\Ee$ is locally free, it fits into a Mayer-Vietoris exact sequence
\begin{equation}\label{eq+s0}
0 \to \Ee \to \Ee _{|H_1}\oplus \Ee _{|H_2}\to \Ee _{|L} \to 0
\end{equation}
and so it is sufficient to prove that at least one among $\Ee _{|H_1}$ and $\Ee _{|H_2}$ splits, due to Lemma \ref{rrr0} and Remark \ref{rrr1}. Up to a twist we may assume that $\Ee _{|L} \cong \Oo _L({c})\oplus \Oo _L$ for some integer $c\ge 0$. From $H^1(\Ee (-c))=0$ we see that at least one among $\Ee _{|H_1}(-c)$ and $\Ee _{|H_2}(-c)$ has a nonzero global section, say there exists a nonzero section $\sigma \in H^0(\Ee _{|H_1}(-c))$. Then $\sigma$ induces an exact sequence on $H_1$:
\begin{equation}\label{eq+s1}
0 \to \Oo _{H_1}({c})\to \Ee _{|H_1}\to  \Ii _{Z,H_1}\to 0
\end{equation}
with $Z$ a zero-dimensional scheme. If $Z=\emptyset$, then $\Ee _{|H_1}$ splits and so we get the assertion. Now assume $Z\ne \emptyset$ and so we get $h^1(H_1,\Ee _{|H_1}(-1))\ne 0$ from (\ref{eq+s1}). But we have $H^1(L,\Ee _{|L})(-1)) =0$ and so (\ref{eq+s0}) gives $H^1(\Ee (-1))\ne 0$, a contradiction.
\end{proof}

Now we pay our attention to aCM bundles on higher dimensional quadrics. Let $X_n =H_{1,n}\cup H_{2,n}\subset \PP^{n+1}$ with $n\ge 3$, be a reducible quadric hypersurface, i.e. $H_{i,n}\subset \PP^{n+1}$ are hyperplanes such that $H_{1,n}\ne H_{2,n}$. Set $L_n:=H_{1,n} \cap H_{2,n}$.

\begin{theorem}\label{s+2}
There is no indecomposable aCM bundles of rank two on $X_n$ for $n\ge 3$.
\end{theorem}

\begin{proof}
Let $\Ee$ be an aCM bundle of rank two on $X_n$. By Lemma \ref{ee1} it is sufficient to prove the existence of integers $a, b$ such that $\Ee _{|H_{i,n}} \cong \Oo _{H_{i,n}}(a)\oplus \Oo _{H_{i,n}}(b)$ for each $i=1,2$.

Let $V\subset \PP^{n+1}$ be a $3$-dimensional linear subspace with $X_n\cap V$ a reducible quadric surface. Since $\Ee _{|V\cap X_n}$ is an aCM bundle of rank two, Proposition \ref{s+1} gives the
existence of $i\in \{1,2\}$ such that $\Ee _{|H_{i,n}\cap V}$ splits, say $\Ee _{H_{i,n}\cap V} \cong \Oo _{H_{i,n}\cap V}(a)\oplus \Oo _{H_{i,n}\cap V}(b)$. By \cite[Theorem 2.3.2 in Chapter II]{oss}, we get $\Ee _{|H_{i,n}} \cong \Oo _{H_{i,n}}(a)\oplus \Oo _{H_{i,n}}(b)$ and in particular we get $\Ee _{|L_n} \cong \Oo _{L_n}(a)\oplus \Oo _{L_n}(b)$. Since $n$ is at least three, we also get $\Ee _{|H_{3-i,n}} \cong \Oo _{H_{3-i,n}}(a)\oplus \Oo _{H_{3-i,n}}(b)$ by \cite[Theorem 2.3.2 in Chapter II]{oss}, concluding the proof.
\end{proof}

\section{Examples}
In this section we introduce two examples of non-splitting aCM sheaf of rank two on $X$ in Example \ref{aa1} and Example \ref{yy1}.

\begin{remark}\label{aa2}
Note that any extension of an aCM sheaf by another aCM sheaf is again aCM. Since the sheaves $\Oo_{H_1}$ and $\Oo_{H_2}$ are aCM, but not locally free, for each integer $r\ge 1$ there exist non-locally free pure aCM sheaves with pure dimension $2$ and with pure rank $r\ge 1$, i.e. rank $r$ at all points of $X\setminus L$.
\end{remark}

\begin{example}\label{aa1}
For a fixed point $p\in H_1\setminus L$, consider the following extension
\begin{equation}\label{eqb1}
0 \to \Oo _X \to \Ee\to \Ii _p(1)\to 0.
\end{equation}
Such $\Ee$ is uniquely determined by $p$, because $\Ext^1 (\Ii_p(1), \Oo_X) \cong H^1(\Ii_p(-1))^\vee$ is $1$-dimensional from the standard exact sequence for $p\in X$ tensored by $\Oo_X(-1)$. Since $h^1(\Oo _X) =0$ and both $\Oo_X$ and $\Ii _p(1)$ are spanned, so $\Ee$ is spanned. We also have $h^0(\Ee)=4$. Let us see now that $\Ee$ is aCM; indeed, it is Ulrich. In order to see this, let us consider the minimal $\Oo_{\PP^3}$-resolution of $\Ii_{p,X}$. Without loss of generality, let $\Ii_p=(x,z,w)$ and then we get
\begin{equation}
0\to \Oo_{\PP^3}(-2)\to\Oo_{\PP^3}(-1)^{\oplus 4} \stackrel{\mathrm{M}}{\to} \Oo_{\PP^3}^{\oplus 3}\to \Ii_{p,X}(1) \to 0,
\end{equation}
with
$$\mathrm{M}=  \left (
                                           \begin{array}{llll}
                                             y & z & w & 0\\
                                             0 & -x& 0 & w\\
                                             0 & 0 & -x & -z
                                           \end{array}
                                         \right).
$$
\noindent Now applying the Horseshoe Lemma to (\ref{eqb1}), we get a resolution
\begin{equation}
0\to\Oo_{\PP^3}(-1)^{\oplus 4}\stackrel{\mathrm{N}}{\to}\Oo_{\PP^3}^{\oplus 4}\to\Ee\to 0
\end{equation}
\noindent given by the matrix
$$\mathrm{N}=  \left (
                                           \begin{array}{llll}
                                             y & z & w & 0\\
                                             0 & -x& 0 & w\\
                                             0 & 0 & -x & -z\\
                                             0 & 0 & 0 & y
                                           \end{array}
                                         \right),
$$
\noindent namely, $\Ee$ has a presentation by a square matrix with linear entries such that $\det \mathrm{N}=(xy)^2$. This is one of the equivalent definitions of $\Ee$ being an Ulrich sheaf of rank two. Since the matrix $\mathrm{N}$ has rank two on $X$, the sheaf $\Ee$ is locally free.

From (\ref{eqb1}) we get that $\Ee$ is $0$-regular and
$$\Ee_{|H_1} \cong TH_1(-1) \text{ and } \Ee_{|H_2}\cong \Oo_{H_2}(1)\oplus \Oo_{H_2},$$
where $TH_1$ is the tangent bundle of $H_1$. In particular, we have $h^0(\Ii _{H_1}\otimes \Ee ) =1$, since $\Ii_{H_1} \cong \Oo_{H_2}(-1)$. So for each $q\in H_1\setminus L$ there is a section $\sigma$ of $\Ee$ vanishing at $q$, but not vanishing identically on $H_1$. Since $\sigma_{|L}\ne 0$, so $q$ is the only zero of $\sigma$. Hence varying $q\in H_1\setminus L$ we get a unique bundle, say $\Ee_1$, up to isomorphisms.

By applying the same argument to $H_2$, we obtain another bundle, say $\Ee_2$. We get $\Ee_1 \not \cong \Ee_2$, because they are exchanged by any automorphism of $X$ exchanging $H_1$ and $H_2$ with fixing $L$ pointwise.

Indeed, each $\Ee_i$ is the restriction of a spinor bundle on a $4$-dimensional smooth hyperquadric $Q_4\subset \PP^5$ to its intersection with $3$-dimensional linear subspace so that the intersection is a union of two planes. We prove in Proposition \ref{tty1} that the three descriptions of $\Ee_1$ and $\Ee_2$, as extensions, the matrix $\mathrm{M}$ and the restriction of the spinor bundles, give the same sheaf.
\end{example}

\begin{example}\label{yy1}
Fix two integers $k,c$ such that $0\le k<c\le 2k+2$. Let $Z\subset L$ be a zero-dimensional subscheme with $\deg (Z) =c-k$ and then we have
$$h^0(\Ii _{Z'}(c-2k-3)) = h^0(\Oo _{H_2}(c-2k-3))=h^0(\Ii _{Z}(c-2k-3))=0$$
for any subscheme $Z'\subset Z$ with colength one. Thus the Cayley-Bacharach condition is satisfied and so there exists a vector bundle $\Gg=\Gg _{c,k,Z}$ on $H_2$ fitting into the following exact sequence
\begin{equation}\label{eqyy1}
0 \to \Oo _{H_2}(k)\to \Gg \to \Ii _Z(c-k) \to 0.
\end{equation}
Since $k\ge 0$ and $\Ii _Z(c-k)$ is globally generated, so $\Gg$ is also globally generated. By tensoring (\ref{eqyy1}) with $\Oo_L$, we get $\Gg _{|L} \cong  \Oo _L({c})\oplus \Oo _L$ and the exact sequence
\begin{equation}\label{eqyy1-1}
0 \to \Oo _{L}(k)\to \Gg _{|L}\to \Ii _Z(c-k)\otimes \Oo _L \to 0,
\end{equation}
in which the first map is injective, because it is injective outside $Z$ and $\Oo _L(k)$ is locally free. For an integer $t\in \ZZ$, let us consider the following natural maps:
\begin{align*}
a_t &: H^1(\Gg(t)) \to H^1(\Ii_Z(c-k+t)),\\
b_t &: H^1(\Gg(t)) \to H^1(\Gg(t)_{|L}),\\
c_t &: H^1(\Ii_Z(c-k+t)) \to H^1(\Ii_Z(c-k+t) \otimes \Oo_L).
\end{align*}
The restriction maps $b_t$ and $c_t$ are maps from cohomology groups appearing in the long exact sequences of cohomology of (\ref{eqyy1}) and (\ref{eqyy1-1}) and we may see them as vertical maps among the vector spaces of these cohomology exact sequences of  (\ref{eqyy1}) and (\ref{eqyy1-1}). We get vertical maps between the corresponding first cohomology groups on $H_2$ and $L$. Since $h^1(\Oo _{H_2}(k+t)) =0$, so $a_t$ is injective for all $t\in \ZZ$. The sheaf $\Ii _Z(c-k+t)\otimes \Oo_L$ has a torsion subsheaf, say $\tau$, of degree $c-k$ and we get that $\left (\Ii _Z(c-k+t)\otimes \Oo_L \right)/\tau \cong \Oo _L(t)$. Since $H^1(\tau)=0$, so we may consider the map $c_t$ as a map
$$d_t : H^1(\Ii_Z(c-k+t)) \to H^1(\Oo_L(t)).$$
For any zero-dimensional subscheme $\tau \subset H_2$, let $\mathrm{Res}_L(\tau)$ denote the residual scheme of $\tau$ with respect to $L$, i.e. the closed subscheme of $H_2$ with $\Ii _\tau:\Ii _L$ as its ideal sheaf. Then we have $\deg (\tau)=\deg (\mathrm{Res}_L(\tau ) ) +\deg (L\cap \tau) $ and for each $t\in \ZZ$ we have an exact sequence of sheaves on $H_2$:
\begin{equation}\label{eqoo1.11}
0 \to \Ii _{\mathrm{Res}_L(\tau ),H_2}(t-1) \to \Ii _\tau (t) \to \Ii _{\tau \cap L ,H_2}(t)\to 0.
\end{equation}
Since $Z\subset L$ is of degree $c-k$, so we get $\mathrm{Res}_L(Z) =\emptyset$ and (\ref{eqoo1.11}) induces an exact sequence
$$0\to \Oo_{H_2}(c-k+t+1) \to \Ii_Z(c-k+t) \to \Oo_L(t) \to 0.$$
We get the injectivity of $d_t$ from $h^1(\Oo _{H_2}(c-k+t-1)) =0$. By the construction of the residual sequence, $d_t \circ a_t$ is the composition of $b_t$ and the projection to $H^1(\Oo_L(t))$, so $b_t$ is injective for all $t\in \ZZ$. By Lemma \ref{x2}, the sheaf induced by $\Oo_{H_1}(c)\oplus \Oo_{H_1}$ and $\Gg$ via (\ref{eqx1}) is aCM.
\end{example}

For a coherent sheaf $\Ee$ on $X$, we have a pair of two integers $(d_1, d_2)$ such that $d_i=c_2(\Ee_{|H_i}^{\circ})$ the $2$nd Chern class of $\Ee_{|H_i}^{\circ}$ for $i=1,2$, where $\Ee_{|H_i}^{\circ}$ is the quotient of $\Ee_{|H_i}$ by its torsion.

\begin{theorem}\label{i4}
For a fixed integer $m\ge 0$, there are triple of integers $(c,d_1, d_2)$ such that there is no algebraic scheme $W$ with $\dim (W) \le m$ with the following property satisfied:

\quad{$(\clubsuit)$} there exists a flat family $\{\Ee _w\}_{w\in W}$ of coherent sheaves of rank two on $X$ such that every aCM sheaf $\Ee$ of rank two on $X$ with
$$c_1(\Ee )=\Oo_X(c) ~\text{ and }~ c_2(\Ee _{|H_i}^{\circ}) =d_i \text{ for }i=1,2$$
is isomorphic to $\Ee _w$ for some $w\in W$.
\end{theorem}

\begin{proof}
Up to twist it is sufficient to consider algebraic families of $0$-regular, but not $(-1)$-regular aCM sheaves with the property that there exists $i\in \{1,2\}$ such that $\Ee _{|H_i}^{\circ} \cong \Oo _{H_i}({c})\oplus \Oo _{H_i}$ and so $d_i=c_2(\Ee _{|H_i}^{\circ}) =0$.

Assume $d_1=0$ and take $d_2=k(c-k)+c-k$ with $c\le 2k$ to adopt aCM sheaves in Example \ref{yy1}. Since $c\le 2k$, we have $h^0(\Gg _{c,k,Z}(-k)) =1$ and so the sheaf $\Gg _{c,k,Z}$ uniquely determines $Z$. Thus we find a family of dimension $c-k$, and so it is sufficient to take $c$ and $k$ with $c \le 2k$ and $c-k >m$.
\end{proof}


\section{Classification}

It turns out from Theorem \ref{i1} that the sheaves in Example \ref{aa1} and Example \ref{yy1} are the only possibility for indecomposable aCM kernel sheaves of simple type with rank two on $X$ up to twist.

\begin{theorem}\label{i1}
Up to twist, every non-splitting aCM kernel sheaf of simple type with pure rank two on $X$ is either as in Example \ref{aa1} or Example \ref{yy1}.
\end{theorem}
\begin{proof}
It follows from Propositions \ref{rrr1+}, \ref{tty1} and \ref{i4.00}.
\end{proof}

Let $\Ee$ be an aCM kernel sheaf of simple type with pure rank two on $X$, and by a twist we may assume that $\Ee_{|L}\cong \Oo_L(c)\oplus \Oo_L$ for some $c\ge 0$. If $c=0$, then $\Ee$ is trivial by the following.

\begin{lemma}\label{rrr1+}
Let $\Ee$ be an aCM sheaf of pure rank two fitting into an exact sequence
\begin{equation}\label{eqrrr1}
0 \to \Ee \stackrel{j}{\to} \Ee _1\oplus \Ee _2 \to \Ee _{1|L}\to 0
\end{equation}
with each $\Ee _i$ a vector bundle of rank two on $H_i$ and $\Ee _1\cong \Oo _{H_1}^{\oplus 2}$.
Then we have $\Ee\cong \Oo _X^{\oplus 2}$.
\end{lemma}

\begin{proof}
By Lemma \ref{rrr0}, $\Ee$ is spanned outside $L$ and so $\Ee _2$ is spanned at all points of $H_2\setminus L$. Since $\Ee _1\cong \Oo _{H_1}^{\oplus 2}$, the natural restriction map $H^0(H_1,\Ee _1)\rightarrow H^0(L,\Ee _{1|L})$ is an isomorphism and so (\ref{eqrrr1}) gives an isomorphism $j_\ast : H^0(\Ee )\rightarrow H^0(H_2,\Ee _2)$. On the other hand, we have $h^1(\Ee (-1)) =h^1(L,\Ee _{1|L}(-1)) =0$. So (\ref{eqrrr1}) gives $h^1(H_2,\Ee _2(-1)) =0$ and so the restriction map $v: H^0(H_2,\Ee _2)\rightarrow H^0(L,\Ee _{2|L})$ is surjective. To see that the map $v$ is injective and so it is an isomorphism, it would be sufficient to prove that $h^0(\Ee_2(-1)) =0$. Assume $h^0(\Ee_2(-1)) >0$ and set $\delta$ to be the maximal positive integer such that $h^0(\Ee_2(-\delta)) >0$. From $\Ee_{2|L} \cong  \Oo_L^{\oplus 2}$ and $\delta>0$, we get $h^0(\Ee_2(-\delta-1))=h^0(\Ee_2(-\delta)) >0$, contradicting the definition of $\delta$. In particular, we get $h^0(\Ee _2) =2$ and so $\Ee _2\cong \Oo _{H_2}^{\oplus 2}$. From (\ref{eqrrr1}) we get $h^0(\Ee )=2$ and that $\Ee$ and $\Oo _X^{\oplus 2}$ have the same Hilbert polynomial. The map $\rho: H^0(\Ee )\otimes \Oo _X \rightarrow \Ee$ is injective, because $\Oo _X$ has no nilpotent and $\rho$ is an isomorphism at all points of $X\setminus L$. Since $\Oo _X^{\oplus 2}$ and $\Ee$ have the same Hilbert polynomial, $\rho$ is an isomorphism.
\end{proof}

Now we deal with the case $c>0$.

\begin{proposition}\label{tty1}
Let $\Ee$ be an aCM sheaf of pure rank two, fitting into the exact sequence
\begin{equation}\label{eq+a2.1}
0 \to \Ee \to \Ee _i\oplus \Ee _{3-i} \to \Oo _L({1})\oplus \Oo _L\to 0
\end{equation}
with $\Ee _i\cong  \Oo _{H_i}(1)\oplus \Oo _{H_i}$ for some $i\in \{1,2\}$. Then either $\Ee \cong \Oo _X(1)\oplus \Oo _X$ or $\Ee$ is as in Example \ref{aa1}.
\end{proposition}

\begin{proof}
With no loss of generality we may take $i=1$. By Lemma \ref{rrr0} the sheaf $\Ee$ is spanned at all points of $X\setminus L$ and so $\Ee_2$ is spanned at all points of $H_2\setminus L$. By the definition of kernel sheaf, we have $\Ee _{2|L} \cong \Ee _{1|L}$ and so $\Ee _{i|L} \cong \Oo _L(1)\oplus \Oo _L$. Thus the general splitting type of $\Ee _2$ is $(1,0)$ and $c_1(\Ee _2) =1$.

First assume that $\Ee _2$ is semistable and so it is stable; $c_1(\Ee _2)$ is odd. Since $h^0(\Ee _2)>0$, we get an exact sequence
$$0 \to \Oo _{H_2}\to \Ee _2\to \Ii _{Z, H_2}(1)\to 0$$
with $Z$ a nonempty zero-dimensional scheme. Thus we get $h^0(\Ee _2)\le 3$ with strict inequality unless $Z$ is a point. Since $\Ee_2$ is globally generated outside $L$, we have $h^0(\Ee_2)\ge 2$. On the other hand, from (\ref{eq+a2.1}) we have $h^1(\Ee_2(-1))=0$ and so $Z$ is a point, i.e. $c_2(\Ee _2)=1$. The classification of stable vector bundles on $\PP^2$ with $(c_1, c_2)=(1,1)$ gives $\Ee _2\cong TH_2(-1)$; see \cite{vdv}.

Now assume that $\Ee _2$ is not semistable and so it fits into an exact sequence
$$0\to \Oo _{H_2}(k)\to \Ee_2 \to \Ii _{Z', H_2}(1-k)\to 0$$
with $Z'$ a zero-dimensional scheme and $k>0$. Since $\Ee _2$ is globally generated outside $L$, we get $k=1$ and $Z'=\emptyset$. Thus we get $\Ee _2\cong \Oo _{H_2}(1)\oplus \Oo _{H_2}$. Now from (\ref{eq+a2.1}) and $h^1(\Ee(-1))=0$, we have $h^0(\Ee(-1))=1$. It gives the nonzero map $\Oo _X({1})\rightarrow \Ee$ with cokernel $\Aa$ of pure rank one and $h^0(\Aa)=1$, since $h^1(\Oo _X({1})) =0$. Since $\Ee$ is spanned at all points of $X\setminus L$, the sheaf $\Aa$ is also spanned outside $L$ and so $\Aa \cong  \Oo _X$. Thus we get $\Ee \cong \Oo _X({1})\oplus \Oo _X$.
\end{proof}

\begin{proposition}\label{i4.00}
Let $\Ee$ be an aCM sheaf of pure rank two, fitting into the exact sequence
\begin{equation}\label{eq+a2.2}
0 \to \Ee \to \Ee _i\oplus \Ee _{3-i} \to \Oo _L({c})\oplus \Oo _L\to 0\end{equation}
with $\Ee _i\cong  \Oo _{H_i}(c)\oplus \Oo _{H_i}$ for some $i\in \{1,2\}$ and $c\ge 2$. Then either $\Ee \cong \Oo _X(c)\oplus \Oo _X$ or $\Ee$ is as in Example \ref{yy1}.
\end{proposition}

\begin{proof}
Without loss of generality we may assume $i=1$, and let $k$ be the maximal integer such that $h^0(H_2,\Ee _2(-k)) >0$. By Lemma \ref{rrr0}, $\Ee_2$ is generated by global sections outside $L$ and so we have $k\ge 0$ with an exact sequence
\begin{equation}\label{eqx3}
0 \to \Oo _{H_2}(k) \to \Ee _2 \to \Ii _{Z, H_2}(c-k)\to 0
\end{equation}
with $Z$ a zero-dimensional subscheme of $H_2$ and $\Ii _{Z, H_2}(c-k)$ globally generated outside $L$. In particular, we have $k\le c$. If $k=c$, we get $Z=\emptyset$ and so $\Ee _2 \cong \Oo _{H_2}({c})\oplus \Oo _{H_2}$. Using the same argument in the last part of proof of Proposition \ref{tty1}, we get $\Ee \cong \Oo _X({c})\oplus \Oo _X$.

\quad {(i)} If $k=0$, then we get $h^0(\Ee _2(-1))=0$ by the definition of the integer $k$. Now from the long exact sequence of cohomology of (\ref{eqx3}) twisted by $\Oo_X(-1)$, we get $\deg (Z)\ge \binom{c+1}{2}$. On the other hand, the following inequality due to Lemma \ref{x2}
$$h^1(\Ee _2(-2))\le h^1(\Ee _2(-2)_{|L}) = h^1(\Oo_L(c-2))+h^1(\Oo_L(-2))=1$$
gives $h^1(\Ii _{Z, H_2}(c-2)) \le 1$, which implies that $\deg (Z)\le 1+ \binom{c}{2}$. So the only possibility is $c=1$, a contradiction.

\quad {(ii)} Now assume $0<k<c$. Since we have $\Ee _{2|L} \cong \Oo _L({c})\oplus \Oo _L$, tensoring (\ref{eqx3}) with $\Oo_L$ gives $\deg (Z \cap L)=c-k$. Set $\phi (t):= h^1(\Ee_2 (t))$ and then by Lemma \ref{x2} with the duality we get
$$\phi (t)  \left\{
                                           \begin{array}{lll}
                                             =0, & \hbox{if $ t\ge -1$;}\\
                                             \le -1-t, & \hbox{if $-c-1 \le t\le -2$;}\\
                                             =0, & \hbox{if $t\le -c-2$.}\\
                                           \end{array}
                                         \right.$$
Twisting (\ref{eqx3}) with $\Oo _{H_2}(-k-e)$ for $e\in \{0,1,2\}$, we get
\begin{equation}
h^1(\Ii _{Z, H_2}(c-2k-e))=   h^1(\Ee_2 (-k-e))  \le k+e-1
\end{equation}
by Lemma \ref{x2}. The same inequality holds also for $e=3$, due to
$$h^1(\Ii_{Z, H_2}(c-2k-3)) \le h^1(\Ee_2 (-k-3))+1 \le k+4.$$
If $c\le 2k+2$, then twisting (\ref{eqx3}) with $\Oo _{H_2}(k-c)$ gives
$$h^1(\Ee_2 (k-c)) = h^1(\Ii _Z) = \deg (Z)-1$$
and so we get $\deg (Z) \le c-k$. In particular we have $\deg (Z) = c-k$ and $Z\subset L$. If $c =2k+2$, then we have $h^0(\Ee_2 (-k-1)) =h^0(\Ii _Z(1)) =1$, contradicting the definition of the integer $k$. If $c\le 2k+1$, then this is as in Example \ref{yy1}.

If $c>2k+2$, then $\Ee_2$ is stable. Since $h^0(\Ee_2 (-k-1)) =0$, so (\ref{eqx3}) gives $h^0(\Ii _{Z, H_2}(c-2k-1)) =0$. Recall that for any zero-dimensional subscheme $\tau \subset H_2$ we get an exact sequence (\ref{eqoo1.11}) of sheaves on $H_2$, where $\mathrm{Res}_L(\tau)$ denotes the residual scheme of $\tau$ with respect to $L$. Since $\deg (Z\cap L) =c-k$, we have $\deg (\mathrm{Res}_L(Z)) =\deg (Z) -c+k$. From (\ref{eqoo1.11}) with $t=c-2k-1$, we get $h^0(\Ii _{\mathrm{Res}_L(Z)}(c-2k-2)) =0$ and so $\deg (\mathrm{Res}_L(Z))\ge \binom{c-2k}{2}$. Thus we get $\deg (Z)\ge \binom{c-2k}{2}+c-k$. On the other hand we also have $\deg (Z)\le \binom{c-2k-1}{2} +k+4$, since $h^1(\Ii _{Z, H_2}(c-2k-3))  \le k+4$. Combining these two inequalities, we get
$$\binom{c-2k}{2}+c-k \le \binom{c-2k-1}{2} +k+4$$
and it implies $c\le 2k+2$, a contradiction.
\end{proof}



\begin{corollary}\label{Ulrich}
The sheaves in Example \ref{aa1} are the only two Ulrich kernel sheaves of simple type on $X$ with rank two, up to twist.
\end{corollary}
\begin{proof}
Let $\Ee$ be a sheaf satisfying the conditions in the assertion and then we have $h^0(\Ee)=4$ by definition. It is clear that the splitting vector bundles are not Ulrich. In Example \ref{aa1} the two sheaves $\Ee_1$ and $\Ee_2$ were shown to be Ulrich.

Assume now that $\Ee$ is a twist of a sheaf $\Ff$ as in Example \ref{yy1}, say $\Ee \cong \Ff(-t+1)$ for some $t\in \ZZ$. Since $\Ff$ is aCM and we have $h^0(\Ee(-1))=0$, so we have
\begin{equation}\label{jk}
h^0(\Ff(-t)_{|H_1}^{\circ})+ h^0(\Ff(-t)_{|H_2}^{\circ})=h^0(\Ff(-t)_{|L}^{\circ}).
\end{equation}
Without loss of generality let us assume $\Ff_{|H_1}^{\circ} \cong \Oo_{H_1}(c)\oplus \Oo_{H_1}$ with $\Oo_X(c)\cong \det (\Ff)$. Then $\Ff_{|H_2}^{\circ}$ fits into (\ref{eqyy1}).

Note that $h^0(\Ff(-t)_{|H_1}^{\circ})=h^0(\Oo_{H_1}(c-t))={c-t+2 \choose 2}$ and it is greater than $h^0(\Ff(-t)_{|L}^{\circ})=c-t+1$ for $t\le c-1$, in which case (\ref{jk}) cannot hold. If $t>c$, then we have $h^0(\Ee)=0$ from (\ref{eqa2}).

Finally assume $t=c$ and then we have $h^0(\Ff(-c))=0$. Since $h^0(\Ee)=h^0(\Ff(-c+1))=4$, we have $h^0(\Ff(-c+1)_{|H_2}^{\circ})=3$. But from the sequence (\ref{eqyy1}) twisted by $\Oo_{H_2}(-c+1)$:
$$0\to \Oo_{H_2}(k-c+1) \to \Ff(-c+1)_{|H_2}^{\circ}\to \Ii_{Z, H_2}(-k+1) \to 0,$$
we get $h^0(\Oo_{H_2}(k-c+1))\le 1$ and $h^0(\Ii_{Z, H_2}(-k+1))\le 1$. In particular we have $h^0(\Ff(-c+1)_{|H_2}^{\circ})\le 2$, a contradiction.
\end{proof}

Now assume $n\ge 3$ and let $\Ee$ be an aCM sheaf on $X_n$ with pure rank $2$. Similary as in the case of $X=X_2$, we say that $\Ee$ is a kernel sheaf of simple type if it fits into an exact sequence
$$0 \to \Ee \to \Ee _1\oplus \Ee_2\to \Ee _{1|L_n}\to 0$$
with $L_n:=H_{1,n} \cap H_{2,n}$ and $\Ee _i$ a vector bundle of rank two on $H_{i,n}$ such that at least one of the two bundles $\Ee_1$ and $\Ee_2$ splits.

\begin{proposition}\label{b1}
Every aCM kernel sheaf of simple type on $X_n$ of rank two for $n\ge 3$ is decomposable.
\end{proposition}

\begin{proof}
Let $\Ee$ be an indecomposable aCM kernel sheaf of simple type of rank two on $X_n$. Up to shift and an exchange of the two components $H_{1,n}$ and $H_{2,n}$ of $X_n$, we may assume that $\Ee _1\cong \Oo _{H_{1,n}}(c)\cup \Oo _{H_{1,n}}$ for some $c\ge 0$.
Let $V\subset \PP^{n+1}$ be a $3$-dimensional linear subspace such that $X:= X_n\cap V$ is a rank two quadric of $V$, i.e. assume $V\nsubseteq X_n$ and $H_{1,n}\cap V \ne H_{2,n}\cap V$. Set $\Ff := \Ee _{|X}$.

Since every germ of $\Ee$ has depth $n$, every germ of $\Ff$ has depth $2$ and the restriction of the kernel sequence defining $\Ee$ is an exact sequence with bundles
$$\Ff_1:= \Ee_{1|H_{1,n}\cap V} \cong \Oo _{H_1}(c)\oplus \Oo _{H_1} \text{ and } \Ff _2 := \Ee _{2|H_{2,n}\cap V}.$$
Since $\Ee$ is aCM, $\Ff$ is also aCM. Since $\Ee$ is indecomposable, we can also obtain that $\Ff$ is indecomposable. By Theorem \ref{i1} either $\Ff$ is as in Example \ref{aa1} with $c=1$, or $c\ge 2$ and there is an integer $k$ with $0<k<c$ and $\Ff$ is as in Example \ref{yy1}.

By Lemma \ref{rrr0} $\Ff$ is globally generated at all points of $X\setminus L$. Take a general section $s\in H^0(\Ff)$. Since $\Ee$ is aCM, it lifts to a nonzero section $\sigma \in H^0(\Ee)$ so that its zeros $\Sigma$ has pure codimension two if it is not empty. Since $\Ee$ does not split, we get that $\Sigma$ is non-empty. Since $n\ge 3$, we have $\Sigma \cap H_{1,n}\cap H_{2,n} \ne \emptyset$ and so $\sigma$ induces a nonzero section $\sigma _1\in H^0(H_{1,n},\Ee _1)$. From $\Ee _1 \cong \Oo _{H_{1,n}}(c)\oplus  \Oo_{H_{1,n}}$ we see that $\sigma _1$ has no zero and so its restriction to $\Ee _{1|H_{1,n}\cap H_{2,n}}$ also has no zero. Let $\sigma_2\in H^0(H_{2,n},\Ee _2)$ be the section induced by $\sigma$. Since $(\sigma _1,\sigma _2)$ are induced by $\sigma$, the kernel sequence of $\Ee$ shows that $\sigma _{1|H_{1,n}\cap H_{2,n}} =\sigma _{2|H_{1,n}\cap H_{2,n}}$ and so $\sigma _2$ has no zero in the hyperplane $H_{1,n}\cap H_{2,n}$ of $H_{2,n}$. Hence $\sigma _2$ has no zero and so $\Ee _2$ splits. Since $\Ee _{1|L_n} \cong \Ee _{2|L_n}$, our aCM sheaf $\Ee$ fits into an exact sequence
\begin{equation}\label{eqer1}
0 \to \Ee \to \Oo _{H_{1,n}}({c})\oplus \Oo _{H_{1,n}}\oplus \Oo _{H_{2,n}}({c})\oplus \Oo _{H_{2,n}}
\to \Oo _{L_n}({c})\oplus \Oo _{L_n}\to 0.
\end{equation}

First assume $c>0$. Since $\Ee$ is aCM, (\ref{eqer1}) gives $h^0(\Ee (-c))=1$ and that the nonzero map $j: \Oo _{x_n}(c) \rightarrow \Oo _{X_n}$ has locally free rank one cokernel at all points of $X_n\setminus L_n$. Recall that any pure sheaf has an Harder-Narasimhan filtration; see \cite[Theorem 1.3.4]{hl}. Since each stalk of $\Ee$ has depth $n$, we see that $j$ is injective and that $\Ee /j(\Oo _{X_n})$ is pure. Since $\Ee$ is aCM, (\ref{eqer1}) gives $h^0(\Ee ) =h^0(\Oo _{X_n}(c)) +1$. Since $j(\Oo _{X_n}(c))$ is aCM, we get $h^0(\Ee/j(\Oo _{X_n})) =1$ and that a nonzero section of $\Ee/j(\Oo _{X_n})$ gives an injective map $u: \Oo _{X_n}\rightarrow \Ee/j(\Oo _{X_n})$, which is an isomorphism outside $L_n$. By (\ref{eqer1}) the sheaves $\Oo _{X_n}$ and $\Ee/j(\Oo _{X_n})$ have the same Hilbert function. Since $u$ is injective we get that $u$ is an isomorphism. Hence $\Ee \cong \Oo _{X_n}({c})\oplus \Oo _{X_n}$.

Now assume $c=0$. From (\ref{eqer1}) we get $h^0(\Ee )=2$ and that the natural map $v: H^0(\Ee )\otimes \Oo _{X_n}\rightarrow \Ee$ is an isomorphism at a general point of $H_{1,n}$ and at a general point of $H_{2,n}$. Thus $v$ is injective with $\mathrm{coker}(v)$ supported on $L_n$. Since (\ref{eqer1}) implies that $\Ee$ and $\Oo_{X_n}^{\oplus 2}$ have the same Hilbert function, $v$ is an isomorphism.
\end{proof}

\providecommand{\bysame}{\leavevmode\hbox to3em{\hrulefill}\thinspace}
\providecommand{\MR}{\relax\ifhmode\unskip\space\fi MR }
\providecommand{\MRhref}[2]{%
  \href{http://www.ams.org/mathscinet-getitem?mr=#1}{#2}
}
\providecommand{\href}[2]{#2}

\end{document}